\documentclass[12pt]{amsart}
\usepackage{hyperref}

\usepackage{amsmath,amsfonts,amssymb}
\usepackage{mathrsfs}
\usepackage{verbatim}
\usepackage{amsthm}
\usepackage{mathrsfs}
\usepackage{xcolor}

\newtheorem{theorem}{Theorem}[section]
 
 \newtheorem{lemma}[theorem]{Lemma}
 \newtheorem{proposition}[theorem]{Proposition}
 \theoremstyle{definition}
 \newtheorem{definition}[theorem]{Definition}
 \theoremstyle{remark}
 
  \newtheorem{ex}[theorem]{Example}
 \numberwithin{equation}{section}

\def \bA {\mathbb A}

\def \bC {\mathbb C}
\def \bD {\mathbb D}

\def \bN {\mathbb N}

\def \bR {\mathbb R}

\def \cL {\mathcal L}
\def \cM {\mathcal M}

\def\supp{{{\rm supp}}}

\begin{document}

\title[On Sobolev regularity]
{$\bC$-elliptic operators and $\mathrm{W}^{1,1}$-regularity\\ for linear growth functionals}
\author
{Piotr Wozniak}

\date{}

\newcommand{\Addresses}{{
		\bigskip
		\footnotesize
		
		\textsc{P. Wozniak}, \textsc{Westf\"alische Wilhelms-Universit\"at M\"unster,
			Institut für Analysis und Numerik, 
			Einsteinstraße 62,
			48149 M\"unster, Germany}\par\nopagebreak
		\textit{E-mail address}: \texttt{pwozniak@uni-muenster.de}
		
}}

\begin{abstract}
In this paper we prove the higher Sobolev regularity of minimisers for convex integral functionals evaluated on linear differential operators of order one. This intends to generalise the already existing theory for the cases of full and symmetric gradients to the entire class of $\bC$-elliptic operators therein including the trace-free symmetric gradient for dimension $n \geq 3$. 
\end{abstract}

\maketitle

\tableofcontents

\section{Introduction and an overview of the problem}
\noindent
Many boundary value problems originating from the theory of plasticity that concern kinetic and dynamic states can be usefully interpreted via variational methods. Determining equilibrium condition of elastoplastic bodies or examining the displacement fields are in fact often regarded as minimisations of suitable integral functionals acting on the full or symmetric gradients of the competitor maps. For instance a thorough exposition discussing such formulations from a fluid dynamics perspective is given by M. Fuchs and G. Seregin in \cite{Fuchs2000}. The numerous physical motivations have initiated much interest in studying the regularity theory of minimisation problems considered over the classes of vector fields of bounded variation and bounded deformation. In actuality this has been a very active and fruitful field of calculus of variations in recent years for which we are going to highlight a fraction of related publications on the topic. In this paper we wish take on a concrete approach to elaborate the notion of higher Sobolev regularity for a wide spectrum of elliptic problems within a general framework and thereafter unify a large number of the already present contributions of the kind. Heuristically our analysis shall be built upon a designated setting which we point out below. 
\subsection{Variational problem and differential operators}
Let $n \geq 2$ be an integer and let $\Omega \subset \bR^n$ be a bounded domain i.e. open, connected and with Lipschitz boundary. For a pair of finite dimensional vector spaces $V, W$ let $f:W \rightarrow \bR$ be a convex integrand and moreover take $\bA$ to be a linear differential operator of order one; this shall be understood vectorially as a mapping from $V$ to $W$, that is, 
\begin{equation*}
\bA = \sum_{\alpha=1}^{n} \bA_{\alpha} \partial^{\alpha}.
\end{equation*} 
Here each $\bA_{\alpha} \in \cL(V, W)$ is a fixed linear map. Let us assume in addition the dimensions of $V$ and $W$ to be at least two. Our discourse is going to be broadly devoted to investigate minimisation problems for a selected variety of autonomous integral functionals taking the form 
\begin{equation}\label{Dirprob}
	\mathscr{F}[u] = \int_{\Omega}^{} f(\bA u) \hspace*{0.1cm}\mathrm{d} x, \hspace{1cm} \text{for $u \in \mathscr{D}_{u_0}$}
\end{equation}
where the competitor set $\mathscr{D}_{u_0} := u_0 + \mathrm{W}^{\bA,1}_0(\Omega)$ is a Dirichlet class of a rightly assigned \textit{generalised} Sobolev space and $\bA$ is $\bR$-\textit{elliptic}. The ellipticity is understood by the symbol map $\bA[\xi]$ being injective as a linear map from $V$ to $W$ for all $\xi \in \bR^n \setminus \{0\}$. The collection $\mathrm{W}^{\bA,1}_0(\Omega)$ denotes the closure of $u \in C^{\infty}_c(\bR^n; V)$ under the seminorm $\|\bA(\cdot)\|_{\mathrm{L}^1(\Omega)}$. A more precise definition of such spaces will be given in the due course. The prime example that illustrates the situation is the area functional $\sqrt{1 + |\cdot |^2}$ acting on $\mathscr{D}_{u_0} = \mathrm{W}^{1, 1}_{u_0}(\Omega; \bR^n)$. Interpreted in accord with (\ref{Dirprob}) this gives the non-parametric minimal area problem of hypersurfaces encircled by the graphs of $\mathrm{W}^{1, 1}$-maps, one of the subjects contained in the paper by L. Beck and T. Schmidt \cite{Beck2013}. In short, there the authors focus on studying the minima of a relaxed version of the boundary value problem augmented to the BV setting. Adapting the integral representation of the Lebesgue-Serrin type extension as in \cite{Anzellotti1985, Goffman1964, Giaquinta1979}, the crucial part is devoted to a derivation of $\mathrm{W}^{1,1}$-regularity of the corresponding generalised minimisers. Interestingly for the value $a = 3$ the underlying method relies on proving uniform bounds of $L\log^2 L$ kind on the full gradients of perturbed minimising sequences obtained earlier by an implementation of the Ekeland variational principle. Similarly to the model minimal area example in our specimen we shall demand the integrand $f: W \rightarrow \bR$ to be of \emph{linear growth}. Indeed this very condition gradually makes the entire minimisation problem very complex in nature when compared to higher power growths since it emits two major obstructions for which the classical techniques of calculus of variations appear insufficient to facilitate. Firstly, in contrast with the scenario in which $f$ is of \emph{$p$-growth} for $p > 1$, that is, some constants $\alpha_1, \alpha_2, \beta>0$ it satisfies $\alpha_1|z|^p- \beta \leq f(z) \leq \alpha_2(1 + |z|^p)$ for all $z \in W $ and with the operator $\bA$ being $\bR$-\textit{elliptic}, one may invoke the Korn-type inequalities discussed in Section \ref{Fnsp}:
\begin{equation}
\int_{\Omega}^{} |\nabla u|^p \hspace*{0.1cm}\mathrm{d} x \lesssim \int_{\Omega}^{} |\bA u|^p \hspace*{0.1cm}\mathrm{d} x \hspace*{1cm} \forall u \in C^{\infty}_c(\bR^n; V).
\end{equation}
Therefore one readily concludes that the functional $\mathscr{F}$ is well-defined and coercive on the Dirichlet class $\mathrm{W}^{1, p}_{u_0} (\Omega; V)$, thus making it possible to work from here onwards by employing the direct method. Although, in the present circumstances one cannot in fact propose the class $\mathrm{W}^{1, 1}(\Omega; V)$ in the place of competitors because of the many potential choices of the operator $\bA$. More explicitly the \emph{linear growth} $p=1$ does not always permit such reductions due to the presence of the so called Ornstein's Non-Inequality \cite{Ornstein1962, Kirchheim2011, Kirchheim2016} which precisely says that: If there is a constant $C >0$ such that 
\begin{equation*}
\int_{\Omega}^{} |\nabla u| \hspace*{0.1cm}\mathrm{d} x \leq C \int_{\Omega}^{} |\bA u| \hspace*{0.1cm}\mathrm{d} x \hspace*{1cm} \forall u \in C^{\infty}_c(\bR^n; V),
\end{equation*}
then $\bA$ and $\nabla$ depend linearly as operators. The difficulty with assembling the strategy is clearly evident for the symmetric gradient $\varepsilon=\frac{1}{2}(\nabla+\nabla^{\mathrm{T}})$, see \cite{Gmeineder2016} describing the setting of BD.\\
The other confronting aspect is that the spaces $\mathrm{W}^{\bA, 1}(\Omega)$ are in general not reflexive which in turn implies compactness issues when inspecting the minimising sequences. This can be resolved usually by considering an augmented form of the functional and allow a larger function class of bounded $\bA$-variation denoted $\mathrm{BV}^{\bA}(\Omega)$. Formally it encompasses those $u \in \mathrm{L}^1(\Omega)$ for which the \emph{$\bA$-gradient} $\bA u$ is distributionally portrayed by finite Radon measures. There have been a number of contributions to this type of problems principally for the full gradient case by Bildhauer and already mentioned Beck, Schmidt \cite{Bildhauer2002, Beck2013, Bildhauer2003, Bildhauer2003D} but also have been investigated by Fuchs and Mingione \cite{Bildhauer2003c, Fuchs2000a} for so called \emph{a-elliptic} convex integrands with the $\mu=a$ convention of Bildhauer and Fuchs, see Section \ref{intro}. In brief and holistic terms what is known up to now in the full gradient case thus with adopted BV spaces, is that the generalised minima exist and belong to $\mathrm{W}^{1, p}_{\mathrm{loc}}(\Omega; \bR^n)\cap\mathrm{W}^{1,1}(\Omega; \bR^n)$ regularity for some $p>1$ within ellipticity regimes up to $a<1+\frac{2}{n}$. The $\mathrm{W}^{1,1}$ regularity for $1+\frac{2}{n}\leq a\leq3$ holds true if additional assumption of local boundedness of such minimisers is added \cite{Bildhauer2002}. In some circumstances this may be extracted from maximum principles or Moser-type arguments. Having said that, exceeding the ellipticity values of $a=3$, at least for Dirichlet problems, even the $\mathrm{W}^{1,1}$-regularity is not completely known to hold. As for other operators in view of the Ornstein's Non-inequality the distributional gradients of $\mathrm{BV^{\bA}}(\Omega)$-maps may not exist as finite Radon measures and so the induced function spaces happen to be strictly larger than $\mathrm{BV}(\Omega)$. In consequence this prevents immediate generalisations and demands a entirely separate argument when passing from the full gradient to certain other differential operators. Nevertheless more recently the mirrored conclusions on the Sobolev regularity have been established also for the symmetric gradient with consubstantially utilised BD setting, see late results by Gmeineder and Kristensen \cite{Gmeineder2020BD, Gmeineder2019}. However with that being mentioned, to the author's awareness thus far any generalisations to a larger families of differential operators are incomplete and have not been coherently categorised, which motivates intentions and the theme of this paper. 
\subsection{Main Objective and Strategy} 
Our attempt therefore is to provide a partial classification of higher Sobolev regularity results to a whole class of operators commonly sharing the $\bC$-ellipticity property. This essentially asks for the operator to have a finite-dimensional nullspace, rigorous definition in Section \ref{section2.2}. Interestingly our investigations will reflect the advances for BV and BD and eventually expand the canon of regularity to a number of broadly studied differential operators such as the trace-free symmetric gradient. Amusingly as in those examples any regularity results largely depend on the elliptic degeneracy of the integrands. The leading goal is to formulate, with the same defining scales of $a$-$elliptic$ integrands as in \cite{Gmeineder2019} for the BD case, the analogous notion of generalised minimisers of the problem (\ref{Dirprob}) and establish their  $\mathrm{W}^{1,p}_{\mathrm{loc}}$-regularity for some exponent $p>1$. More precisely the integrands $f$ within our considerations are convex, of \emph{linear growth}, twice continuously differentiable and \emph{a-elliptic} for $1<a<1+\frac{1}{n}$, see Section \ref{intro} for a dedicated discussion on these assumptions and the ellipticity range. As for our strategy, adhering to the main skeleton determined in the BV theory \cite{Beck2013} as well as \cite{Anzellotti1985, Giaquinta1979, Demengel1984} we will contemplate an explicit formula for the functional relaxation crucially treating trace values of the competitors at the boundary of $\Omega$, cf. Section \ref{Sectionrel} for the context. The extension to $\mathrm{BV}^{\bA}$ is of course a functional acting on measures and is built upon the decomposition of a Radon measure into absolutely continuous and singular parts. Also in Section \ref{Sectionrel} we include a description in more technical terms. Morally this requires existence of an operator $\mathrm{Tr}$ such that for a map $u \in \mathrm{BV}^{\bA}(\Omega)$ it satisfies $\mathrm{Tr}(u) \in \mathrm{L}^1(\partial \Omega; \mathrm{d}\mathscr{H}^{n-1})$ essentially mimicking the constructions for $\mathrm{BV}$- and $\mathrm{BD}$-maps. Fortunately the works of \cite{Breit2010, Gmeineder2019a, Gmeinedera} give us precisely such necessary trace characterisation as an equivalent requirement of $\bA$ being $\bC$-elliptic, see the statement of Theorem \ref{thmtrace} below. To address the regularity questions our line of argument is based on the vanishnig viscosity approaches of \cite{Beck2013, Gmeineder2019}. In particular this involves stabilising mechanism through an application of the Ekeland variational principle to obtain a minimising sequence approximating the generalised minimiser in suitably chosen perturbation. With the aid of additional locally uniform $\mathrm{L}^p(\Omega; V)$, $p>1$ bound on the $\bA$-gradients, the Ekeland sequence will then become a good candidate for testing the perturbed Euler-Lagrange equation in the final phase of the proof. Because the induced viscosity sequence belongs to the generalised Sobolev space $\mathrm{W}^{\bA,1}$ it is essential to circumvent the appearance of the difference quotient of measures. We will see in Section \ref{viscsection} that with the given set up a good choice for the perturbation space is the usual Lebesgue space $\mathrm{L}^1(\Omega; V)$, weak enough to extrapolate bounds from the available approximants. At the same time this is certainly not a canonical option as the negative Sobolev space $\mathrm{W}^{-1,1}(\Omega)$ as dealt with in \cite{Beck2013} is also a valid alternative. However in contrast to the $\mathrm{BV}$-regime constituting the second order estimates is not adequate within the constrains of the variational principle (\ref{Dirprob}). Aiming to derive the regularity for all generalised minimisers, an alternative method is suggested that favours exploiting the fractional differentiability properties of $\mathrm{BV}^{\bA}$-maps. These modification are partially due to the fact that elements of the viscosity sequence may not have their full gradients uniformly bounded in $\mathrm{L}^{1}(\Omega; V)$. The idea largely resembles the fractional estimates proposed in \cite{Gmeineder2019} by realising that locally $\mathrm{BV}^{\bA}(\Omega)$ embeds continuously into an appropriate Besov scale. To reveal the very essence of our upcoming analysis let us state the central theorem which displays the Sobolev regularity of the generalised minimisers of the variational problem (\ref{Dirprob}).
\begin{theorem}{\normalfont{(The Main Theorem)}}\label{thmreg}
	Let $f \in C^2(\mathrm{Dom}(f))$ be a convex, $a$-elliptic integrand satisfying the linear growth and such that $1<a<1+\frac{1}{n}$. Then any generalised minimiser of $\mathrm{(\ref{Dirprob})}$ is of class $\mathrm{W}^{1,p}_{\mathrm{loc}}(\Omega; V) \cap \mathrm{W}^{\bA,1}(\Omega)$ for some exponent $p=p(a, n, \bA) > 1$.
\end{theorem}
\noindent 
Here $\mathrm{Dom}(f)$ denotes the domain of $f$ the configuration of which will be settled in the due course. As pointed out by authors in \cite{Gmeineder2019}, which also vastly affects our cases, the reasoning applies to a nonoptimal ellipticity range $1< a <1+\frac{1}{n}$. There have been instances of overcoming this issue such as in \cite{Gmeineder2020BD} for BD by elaborating certain estimates in Orlicz spaces to 'close the gap'. In that instance the structures specific to the symmetric gradient are exploited to develop universal $\mathrm{W}^{1,1}_{\mathrm{loc}}$-regularity estimates which play a quintessential role in the main proof. In such the leading scheme unveils that the final conclusion are drawn by means of reaching uniform second order bounds on the viscosity sequence perturbed in the adjusted space $\mathrm{W}^{-2, 1}(\Omega)$. That being said such type of bounds do not appear to carry over to various other possible operators. Therefore reaching the optimality $a < 1+ \frac{2}{n}$ in a universal manner is yet to be discovered. In summary this paper is in a way an extension of the work done in \cite{Gmeineder2016, Gmeineder2020BD} together with \cite{Beck2013} and generalises the approach of Gmeineder and Kristensen \cite{Gmeineder2019}.
\subsection{Structure of the paper}
In Section \ref{prelim} we provide the main set up by emphasising all the relevant items including the assumptions on $f$, ellipticity of the operator $\bA$, auxiliary results and discussing properties of the generalised Sobolev spaces. Moving on Section \ref{Sectionrel} focuses around defining functional relaxation, the generalised minimisers and concludes by stating the regularity theorem. The final Section \ref{Section4} is entirely devoted to the proof of the main result of this paper.\\
\newline
\text{\textbf{Acknowledgement}}:
The author is sincerely grateful to Dr. Franz Gmeineder for the encouragement to work on the problem and for showing great support.

\section{Preliminaries and general theme}\label{prelim}  
\noindent
\text{\textbf{Notation}}. We give a few short comments on the notation present in the paper for clarity and coherence. Firstly given any $x, y \in \bR^n$ we denote by $\langle \cdot, \cdot \rangle$ the standard euclidean product and $V, W$ finite-dimensional vector spaces of dimension greater than or equal to two. For a map $u: \Omega \subset \bR^n \rightarrow V$ and $e_i \in \bR^n$ a unit vector, the finite differences are going to be written as $\tau^{\pm}_{i,h}u(z) := u(z\pm he_i) - u(z)$ while $\Delta^{\pm}_{i,h}u(z) = \frac{1}{h}(u(z\pm he_i) - u(z))$ the difference quotients whenever defined. In the case when distinction between the two is irrelevant we will omit the $\pm$ signs. The symbols $\mathrm{d}\mathscr{L}^{n}$ and $\mathrm{d}x$ both indicate integration with respect to the $n$-dimensional Lebesgue measure. Importantly often whilst performing calculations involving estimates with constants, despite possibility of varying value we will keep the same letter denoting them unless a substantial dependence occurs which will be indicated. For a differential operator $\bA$ the associated tensor pairing is given by 
\begin{equation*}
v \otimes_{\bA} \xi := \sum_{\alpha=1}^{n} \xi_{\alpha}\bA_{\alpha}v
\end{equation*}
where $\xi \in \bR^n$ and $v\in V$. Given $\eta \in C^1(\Omega)$ and $u \in C^1(\Omega; V)$ with the above pairing there is a counterpart of the \textit{product rule}:
\begin{equation}\label{prdctrul}
\bA(\eta u) = \eta\bA u + u \otimes_{\bA} \nabla \eta.
\end{equation}
In the case of the full gradient operator $\nabla$ (where $\otimes_{\bA}$ is regarded as the generic dyadic product) we recover the usual differentiation rule.  
\subsection{Assumptions on the integrand}\label{intro}
Here we are going to concretely specify the crucial conditions necessarily to be imposed on $f$ in the spirit of the motivations contained in \cite{Gmeineder2019}. These will play a central role in the due regularity argumentation.\\
\newline
	(1) \textit{$f\in C^2(\mathrm{Dom}(f))$ and convex}: Here we declare the domain of $f$ to be the $\bA$-tensor pairing of $V$, that is, $\mathrm{Dom}(f) = \mathrm{span}(V \otimes_{\bA} \bR^n)$. Justifying this seemingly specific choice, we would like to consider the minimal cone within the space $W$ which contains the image of $\bA$ under the space $V$. This is being conducted to match the assertion that the integrand $f$ is meant to act on $\bA u$. After taking the Fourier transform of the latter one deduces the above domain definition of $f$. For instance in the case of the full gradient $\bA = \nabla$, one essentially gets the ordinary dyadic product and work with $\mathrm{Dom}(f) = \bR^{n \times n}$ whereas for the symmetric gradient $\bA= \frac{1}{2}(\nabla + \nabla^{\mathrm{T}})$, such cone is given by $\mathrm{Dom}(f) = \bR^{n \times n}_{\mathrm{sym}}$ the symmetric matrices.\\
	\newline
	(2) \textit{Linear growth condition}: We say that $f: W \rightarrow \bR$ is of linear growth if there exists a positive constant $C>0$ such that
	\begin{equation*}
	|f(z)| \leq C(1+|z|)
	\end{equation*}
	for all $z \in W$.\\
	The growth condition largely makes the entire methodology differ from the $p$-growth integrands for $p>1$. In particular given the latter case we may employ the existent Korn-type inequalities and observe that the resulting functional $\mathscr{F}$ is well-defined and coercive on $\mathrm{W}^{1, p}(\Omega, V)$. Furthermore in comparison with the $L\log L$ growth of $f$, an application of singular integral theory tells us that $\bA u\in L \log L$ implies $\nabla u \in \mathrm{L}^1$ and this accentuates the borderline case beyond which we cannot extract more information on the full gradients. Thus the linear growth provokes the competitors to belong to a relevant function space and retain the functional's coercivity and this to a degree explains the importance of the $\mathrm{W}^{\bA, 1}$-spaces.\\ 
	\newline
	(3) \textit{ the scale of $a$-ellipticity for $1 < a < \infty$}: In the spirit of of Bernstein's work on minimal surface problems \cite{Bernstein1912} the current interpretation of this notion is attributed to Bildhauer and Fuchs preceded by similar considerations of Ladyzhenskaya and Uraltseva \cite{Ladyzhenskaya1970}. A function $f \in C^2(\mathrm{Dom}(f))$ is said to be $a$-elliptic if there are constants $c_1 , c_2 > 0$ such that
	\begin{equation*}
	\frac{c_1|M|^2}{(1+|N|^2)^{\frac{a}{2}}} \leq \langle \nabla^2f(N)M, M \rangle \leq \frac{c_2|M|^2}{(1+|N|^2)^{\frac{1}{2}}}
	\end{equation*}
	for all $M, N \in \bR^{n \times n }$. Here we use the $\mu=a$ identification in comparison with the terminology in \cite{Bildhauer2003}. We will specifically focus on the range $1<a <1 + \frac{1}{n}$.\\
	Indeed the reason for excluding the borderline value $a=1$ is due to the consequent $L\log L$ growth of the integrand. To see this, for $f \in C^2(\bR)$ and 1-$elliptic$ a direct calculation shows that the derivative $|f'|$ is unbounded and hence negates the linear growth assumption. Although convexity does give resembling behaviour from below and above however it does not extend to second order terms. Furthermore since as for the upper bound we shall follow the approach of \cite{Gmeineder2019} and bring forward the fractional estimates which comes at the cost of restricting the ellipticity range to at most $a < 1 + \frac{1}{n}$. We should stress that nevertheless the regularity results for the range $a < 1 + \frac{2}{n}$ for all generalised minima have been established in specific cases e.g. F.Gmeineder in \cite{Gmeineder2020BD} in $\mathrm{BD}$ obtains it by elaborating second order uniform bounds. There the full gradients estimates are tailored precisely for the symmetric gradient. On the other hand for $1 + \frac{2}{n} \leq a \leq 3$ it's not totally clear how to obtain $\mathrm{L}^{\infty}$ which would come in vital when deriving the Sobolev regularity of the minimisers.
\subsection{On the differential operator $\bA$}\label{section2.2}
In order to perform the Ekeland type viscosity approximation for which a relaxation of the functional is inevitably needed, we have got to utilise sufficient trace theory of the set of competitors. This heavily relies on the ellipticity sort that the operator $\bA$ displays. Within our interest here is the one stated below:\\
\newline
{$\bC$-\emph{ellipticity condition}\cite{Smith1970, Gmeineder2019a}}:
Let $\bA$ be a linear differential operator from $V$ to $W$. By taking a natural complexification $V + iV$ and $W+iW$ we say that $\bA$ is $\bC$-elliptic if for any $\xi \in \bC^n \setminus \{0\}$ the induced symbol map
	\begin{equation*}
		\bA[\xi] = \sum_{\alpha=1}^{n} \bA_{\alpha} \xi_{\alpha} \in \mathcal{L}(V+iV, W+iW)
	\end{equation*}
	is injective as a linear map. The nullspace of the operator $\bA$ is given by
	\begin{equation*}
		\ker(\bA; \Omega) := \{\Phi \in \mathcal{D}(\Omega; V);\,\bA\Phi = 0\} 
	\end{equation*}
	where $\mathcal{D}(\Omega; V)$ stands for the vector-valued distributions and the equality in the set is understood distributionally. Importantly there is a characterisation which states that:
	\begin{equation} 
	\text{the operator $\bA$ is $\bC$-elliptic} \iff \text{$\dim \ker(\bA; \Omega) < \infty$}.\label{nullspA}
	\end{equation}
\begin{ex}
Let us provide common instances of such operators present in multiple contexts:\\
$\mathrm{(i)}$ Symmetric gradient: This operator is given by $\varepsilon(u):=\frac{1}{2}(\nabla u + (\nabla u)^{\mathrm{T}})$ with its nullspace being affine transformations of anti-symmetric gradient called the space of rigid deformations $\mathcal{R}(\Omega):= \{x \mapsto Ax + b:\,A \in \bR^{n \times n}, \, b \in \bR^n\}$ which is finite-dimensional. Thus by (\ref{nullspA}) the operator $\varepsilon$ is a representative of $\bC$-elliptic family.\\
\newline
$\mathrm{(ii)}$ Trace-free symmetric operator: Let $V = \bR^n$ and $W = \bR^{n\times n}$. We define it by $\mathcal{E}^{D}u:=\varepsilon(u) - \frac{1}{n}\mathrm{div}(u)I_n$ where $I_n$ is the ($n \times n$)-identity matrix. This operator appears in several applications such as general relativity \cite{Bartnik2004} when studying the gravitational fields. This is a degenerate example as the $\bC$-ellipticity is preserved for $n \geq 3$ however when inspecting $n=2$ the issue is well illustrated by contemplating the map $z \mapsto (1-z)^{-1}$ on the disc $\bD \simeq B_1(0)$ with the identification $\bR^2 \simeq \bC$. The authors in \cite{Breit2010} discovered strong links between the ellipticity type of an operator, its nullspace and the traces of function spaces associated to it.
\end{ex}

\subsection{The V-potentials}	
Invoking the constructions credited to Acerbi and Fusco \cite{Acerbi1989}, for a parameter $1<t<2$ we define the function $V_t:\bR^n \rightarrow \bR^n$ by
\begin{equation*}
V_t(z) := (1+|z|^2)^{\frac{1-t}{2}}z.
\end{equation*}
\noindent	
This supplementary map will turn out quite handy for getting an intermediate step when bounding Sobolev norms especially because it satisfies the following:
\begin{itemize}
\item for a measurable $u:\bR^n \rightarrow \bR^n, h>0$ and a unit vector $e_i$
\begin{equation}\label{V1}
	|\tau^{\pm}_{i,h}V_t(u(x))| \sim (1+|u(x+he_i)|^2 + |u(x)|^2)^{\frac{1-t}{2}}|\tau^{\pm}_{i,h}u(x)|
\end{equation}
\item we can find $C>0$ such that 
\begin{equation}\label{V2}
\min\{|z|,|z|^{2-t}\} \leq CV_t(z)
\end{equation}
\item with $u$ as above and $\omega \subset \bR^n$ open, bounded it follows that 
\begin{equation}\label{V3}
\int_{\omega}^{} |u|^{(2-t)p} \hspace*{0.1cm}\mathrm{d}\mathscr{L}^{n} \leq \mathscr{L}^n(\omega) + C_p \int_{\omega}^{} |V_t(u)|^p \hspace*{0.1cm}\mathrm{d}\mathscr{L}^{n} 
\end{equation}
\item for the integrand $f$ as in (*) and $\Omega \subset \bR^n$ bounded and Lipschitz domain there is a constant $C_{t,n} > 0$ in a way that for all $u:\Omega \rightarrow \bR^N$ measurable, all Lipschitz compact $K \Subset \Omega$ and all $|h| < \mathrm{dist}(K, \partial \Omega)$ we have the bound
\begin{equation}\label{V4}
\frac{|\tau^{\pm}_{i,h}V_t(u(x))|^2}{(1+|u(x+he_i)|^2 + |u(x)|^2)^{\frac{2(1-t) + a}{2}}} \leq C\frac{|\tau^{\pm}_{i,h}u(x)|^2}{(1+|u(x+he_i)|^2 + |u(x)|^2)^{\frac{a}{2}}}.
\end{equation} 
\end{itemize}
The justification of the above properties can be also found in \cite{Gmeineder2019}.
\subsection{Function spaces}\label{Fnsp}
Before we fully plunge into the main part of our discourse we define the underlying function spaces tailored to our Dirichlet problem, namely the generalised $\bA$-Sobolev spaces as well as the maps of bounded $\bA$-variation. We will point out the key results by invoking a rigorous analysis contained in \cite{Breit2010}. Before however let us quickly recap on the Besov scales. For $\Omega$ open, $0<s<1$ and $1 \leq p,q \leq \infty$ a function $u: \Omega \rightarrow V$ is in $\mathrm{B}^s_{p,q}(\Omega)$ if $u \in \mathrm{L}^p(\Omega, V)$ and 
\begin{equation*}
\begin{aligned}
&|u|_{\mathrm{B}^s_{p,q}} := \Big(\int_{0}^{\infty} \Big(\frac{\sup_{|h| < t}\|\Delta_{h}u\|_{\mathrm{L}^p(\Omega)}}{h^s}\Big)^q \hspace*{0.1cm} \frac{\mathrm{d}t}{t}\Big)^{\frac{1}{q}} < \infty \hspace*{1cm} \text{provided $p, q \in [1, \infty)$}\\
&|u|_{\mathrm{B}^s_{p,\infty}} := \sup_{|h|>0} \frac{\|\Delta_{h}u\|_{\mathrm{L}^p(\Omega)}}{h^s} \hspace*{5.3cm} \text{provided $p \in [1,\infty)$}
\end{aligned}
\end{equation*}
where $s>0$ and $\Delta_{h}$ represents the finite difference if defined and being equal to zero if $x+h \notin\Omega$. Furthermore the Besov seminorm is equivalent to the Gagliardo seminorm
\begin{equation*}
|u|_{\mathrm{W}^{s,p}} := \int_{\Omega} \int_{\Omega} \frac{|u(x)-u(y)|^p}{|x-y|^{n+sp}} \hspace*{0.1cm} \mathrm{d}x \,\mathrm{d}y.
\end{equation*}
Finally we remark on a counterpart of the Sobolev-type embedding in this setting: for any $1 \leq p < \frac{n}{n-s}$
\begin{equation}\label{embbBes}
(\mathrm{B^s_{1, \infty}})_{\mathrm{loc}}(\bR^n) \hookrightarrow \mathrm{L}^p_{\mathrm{loc}}(\bR^n; V).
\end{equation} 
\begin{definition}
Let $\Omega \subset \bR^n$ be an open set. Then we define the following:\\
$\mathrm{(1)}$ The $\bA$-Sobolev space $\mathrm{W}^{\bA,1}(\Omega)$ is given by all the measurable functions $u:\Omega \rightarrow V$ such that
\begin{equation*}
\|u\|_{\mathrm{W}^{\bA,1}(\Omega)} := \|u\|_{\mathrm{L}^1(\Omega)} + \|\bA  u\|_{\mathrm{L}^1(\Omega)} <\infty
\end{equation*}
where the differential operator is understood in the distributional sense.\\
\vspace{1pt}
\\
$\mathrm{(2)}$ the space of functions of bounded $\bA$-variation $\mathrm{BV}^{\bA}(\Omega)$ which comprises all $u:\Omega \rightarrow V$ such that $\bA u$ is a Radon measure on $\Omega$ and 
\begin{equation*}
\|u\|_{\mathrm{BV}^{\bA}} := \|u\|_{\mathrm{L}^1(\Omega)} + |\bA u|(\Omega) < \infty
\end{equation*}
where the latter term denotes the total variation given by
\begin{equation*}
|\bA u|(\Omega) = \sup \{\int_{\Omega}^{}\langle u, \bA^* \phi \rangle\hspace*{0.1cm}\mathrm{d}\mathscr{L}^{n} \,:\, \phi \in C^1_{c}(\Omega, W)\, , \|\phi\|_{\mathrm{L}^{\infty}} \leq 1\}.
\end{equation*}
where  $\bA^*$ is the formal adjoint. Exemplarily this notation gives $\mathrm{W}^{1, 1} = \mathrm{W}^{\nabla, 1}$, $\mathrm{BV} = \mathrm{BV^{\nabla}}$ and $\mathrm{BD} = \mathrm{BV^{\varepsilon}}$. Elaborating on the latter space this characterisation means
\begin{equation*}
\mathrm{BD}(\Omega) := \{u \in \mathrm{L^1}(\Omega; \bR^n) \,:\, \varepsilon(u) \in \mathcal{M}(\Omega; \bR^{n \times n}_{\mathrm{sym}})\}
\end{equation*}
so in particular
\begin{equation*}
|\varepsilon(u)|(\Omega) := \sup \{\int_{\Omega}^{}\langle u, \mathrm{div}(\phi) \rangle\hspace*{0.1cm}\mathrm{d}x \,:\, \phi \in C^1_{c}(\Omega, \bR^{n\times n}_{\mathrm{sym}})\, , \|\phi\|_{\mathrm{L}^{\infty}} \leq 1\} < \infty.
\end{equation*}
Notice that there is a strict inclusion $\mathrm{BV}(\Omega) \subsetneq \mathrm{BD}(\Omega)$ due to Ornstein's. The space of bounded deformation is instrumental in fluid dynamics and for an incomplete list of where it has been examined we refer the reader to e.g. \cite{Ambrosio1997, Anzellotti1980, Kohn1982, Suquet1979, Temam1980}.
\end{definition}
\noindent
\text{\textbf{Notions of convergence}}. Analogously to the characterisations of functions of bounded variations \cite{Ambrosio2000} we distinguish several topologies on $\mathrm{BV^{\bA}}$ coming from the types of convergences. Each displays useful for our investigations. Let $\Omega \subset \bR^n$ be a bounded domain with Lipschitz boundary $\partial \Omega$. Let $u \in \mathrm{BV^{\bA}}(\Omega)$ and take a sequence $\{u_k\}$ contained in $\mathrm{BV^{\bA}}(\Omega)$. It is said that
\begin{itemize}
\item $u_k$ converges to $u$ \emph{weakly}*, written $u_k \rightharpoonup^{*} u$ if and only if $u_k \rightarrow u$ in $\mathrm{L}^1(\Omega; V)$ and $\bA u_k \stackrel{*}{\rightharpoonup}\bA u$ weakly as Radon measures
\item  $u_k$ converges to $u$ \emph{strictly}, written $u_k \stackrel{d_s}{\longrightarrow} u$ if and only if
\begin{equation*}
\int_{\Omega}^{} |u-u_k| \hspace*{0.1cm}\mathrm{d}\mathscr{L}^{n} + \big||\bA u_k|(\Omega)-|\bA u|(\Omega)\big| \longrightarrow 0.
\end{equation*}
\item $u_k$ converges to $u$ \emph{area-strictly}, written $u_k \stackrel{\langle \cdot \rangle}{\longrightarrow} u$ if and only if $u_k \rightarrow u$ in $\mathrm{L}^1(\Omega; V)$ and
\begin{equation*}
\int_{\Omega}^{} \sqrt{1+|\bA [\nabla] u_k|^2} \hspace*{0.1cm}\mathrm{d}\mathscr{L}^{n} + |\bA^s u_k|(\Omega) \longrightarrow \int_{\Omega}^{} \sqrt{1+|\bA [\nabla] u|^2} \hspace*{0.1cm}\mathrm{d}\mathscr{L}^{n} + |\bA^s u|(\Omega) 
\end{equation*}
where the above expressions incorporate the Lebesgue-Radon-Nikodym decomposition as for the BV-theory namely:
\begin{equation}\label{Lebdec}
\bA u = \bA^a u + \bA^s u = \bA [\nabla] u + \bA^s u
\end{equation}  
with the density $\bA[\nabla] u$ standing for the approximate gradient borrowing the convention from \cite{Gmeineder2019d}.
\end{itemize}
\vspace{10pt}
Let us list a selection of essential properties respected by the above defined function spaces. 
\begin{theorem}\label{compactn}
Let $u,u_1,u_2,\dots \in \mathrm{BV^{\bA}}(\Omega)$ where $\Omega$ is a bounded Lipschitz domain. Then the following hold:\\
$\mathrm{(a)}$ Given $u_k \rightarrow u$ in $\mathrm{L}^1_{\mathrm{loc}}(\Omega; V)$, then
\begin{equation*}
|\bA u|(\Omega) \leq \liminf_{k\rightarrow \infty} |\bA u_k|(\Omega).
\end{equation*}
$\mathrm{(b)}$ If $\{u_k\}$ is uniformly bounded in $\mathrm{BV^{\bA}}$-norm, then there exists $u$ in $\mathrm{BV^{\bA}}(\Omega)$ and a subsequence $\{u_{k_j}\}$ such that $$u_{k_j} \rightharpoonup^{*} u.$$
$\mathrm{(c)}$ Smooth approximation: $(C^{\infty} \cap \mathrm{W}^{\bA,1})(\Omega)$ is a dense subset of $\mathrm{W}^{\bA,1}(\Omega)$ with respect to the norm topology and a dense subset of $\mathrm{BV^{\bA}}(\Omega)$ in the area strict topology.\\
\vspace{1pt}
\\
$\mathrm{(d)}$ There exists a bounded linear extension operator $E_{\Omega}:\mathrm{BV^{\bA}}(\Omega) \rightarrow \mathrm{BV^{\bA}}(\bR^n)$.
\end{theorem}
\noindent
The contributions of Gmeineder \cite{Breit2010, Gmeineder2019a} provide a thorough exposition of the traces of the maps of $\bA$-variation. This is an essential ingredient in the entire strategy since it ultimately allows us conduct very similar method employed in the examples of BV and BD spaces. Here importantly we emphasize the fact that our differential operator has got to be $\bC$-elliptic, since this very condition ensures the same trace properties. For future reference, let us state a few results describing this matter:
\begin{theorem}{\normalfont{(Existence of Traces \cite[Thm 4.17]{Breit2010})}}\label{thmtrace}
For an operator $\bA[D]$ that is $\bC$-elliptic there exists a trace operator $\mathrm{Tr} :\mathrm{BV}^{\bA}(\Omega) \rightarrow \mathrm{L}^1(\partial\Omega)$ such that
\begin{itemize}
\item $\mathrm{Tr}(u) = u|_{\partial \Omega}$ $\mathscr{H}^{n-1}$-a.e on $\partial\Omega$ provided that $u$ is in $C(\overline{\Omega})\cap\mathrm{BV}^{\bA}(\Omega)$ and is a unique such extension
\item $\mathrm{Tr}$ is continuous in the strict topology
\item the restriction $\mathrm{Tr} :\mathrm{W}^{\bA,1}(\Omega) \rightarrow \mathrm{L}^1(\partial\Omega; \mathscr{H}^{n-1})$ is surjective,
\end{itemize}
where elements in the target space are regarded with respect to the Hausdorff measure $\mathscr{H}^{n-1}.$
\end{theorem}
\noindent
As it is customary for the Sobolev functions we define accordingly the \emph{zero boundary values} space as
\begin{equation*}
\mathrm{W}^{\bA ,1}_0(\Omega) := \{u \in \mathrm{W}^{\bA ,1}(\Omega): \mathrm{Tr}(u) = 0\}.
\end{equation*}
This respects an equivalent characterisation given by $\mathrm{W}^{\bA ,1}_0(\Omega) = \overline{C^{\infty}_c(\Omega; V)}^{\|\cdot\|_{\mathrm{W}^{\bA ,1}}}$ i.e. the closure in the $\mathrm{W}^{\bA ,1}$-norm, see \cite{Breit2010}. The trace theorem in consequence allows us to glue two $\mathrm{BV}^{\bA}$ maps together and also incorporates area-strict smooth  approximation.
\begin{theorem}\cite[Cor. 4.21]{Gmeineder2019a}\label{glue}
	Let $\Omega \Subset K$ for both $\Omega$ and $K$ being bounded Lipschitz domains and let $\bA$ be $\bC$-elliptic. Suppose that $u \in \mathrm{BV}^{\bA}(\Omega)$ and $v \in \mathrm{BV}^{\bA}(K \setminus \Omega)$ and define $\tilde{u} := u \sqcup v := \chi_{\Omega}u + \chi_{K \setminus \overline{\Omega}}v$. Then it holds that $\tilde{u} \in \mathrm{BV}^{\bA}(K)$.
\end{theorem}
\begin{theorem}\label{areastr}
Consider the assertions of Theorem \ref{glue} and take $u, v$ in $\mathrm{BV^{\bA}}(K)$ such that $u = u_0$ on $K \setminus \Omega$. Then there exists a sequence $\{u_j\} \subset u_0 + C^{\infty}_c(\Omega; V)$ such that $u_j \rightarrow u$ in the area-strict topology.   
\end{theorem}
\noindent
There is a intriguing question of more general interest going beyond the scope of this discourse whether the $\mathrm{L}^p$ norm of the gradient can be controlled by the $\mathrm{L}^p$ norm of a differential operator $\bA$. For instance globally if $p>1$ and $\bA$ is just $\bR$-$elliptic$ this is true and is a direct consequence of Calder\'on-Zygmund theory \cite{Calderon1952, Stein2016}. However the parallel conclusions in the limiting case $p=1$ are constrained by the previously mentioned Ornstein's non-inequality. In the local regime the presence of generalised Poincar\'e inequalities have been devised and extensively studied dating back among others to Aronszajn \cite{Aronszajn1954}, Smith \cite{Smith1970} or Reshetnyak \cite{Reshetnyak1970} and later elaborated in the weighted setting by Ka\l amajska \cite{Kaamajska1994} with the importance of $\bC$-ellipticity condition on $\bA$ being present in disguise within their works. In the paper \cite{Gmeineder2019a} authors deliver a transparent view on a variety of Poincar\'e type estimates and we give one such variation suited to our setting which is also reminiscent of Korn's inequality for symmetric gradient: 
\begin{proposition}{\normalfont{(Poincar\'e and Korn-type inequalities)}}\label{Poinc}
Let $\bA$ be a $\bC$-elliptic operator of order one and let $\Omega \subset \bR^n$ be a star-shaped domain w.r.t. a ball. Then given $u \in C^{\infty}(\Omega; V)$ such that $\|\bA u\|_{\mathrm{L}^p(\Omega)} < \infty$ there exists an element in the nullspace of $\bA$, $a \in \ker(\bA)$ such that   
\begin{equation*}
\|u - a\|_{\mathrm{L}^p(\Omega)} \leq C_{p, \Omega} \| \bA u \|_{\mathrm{L}^p(\Omega)}
\end{equation*}
where $1 \leq p \leq \infty$. If $1<p<\infty$, then also
\begin{equation*}
\|\nabla(u - a)\|_{\mathrm{L}^p(\Omega)} \leq C_{p, \Omega} \| \bA u \|_{\mathrm{L}^p(\Omega)}.
\end{equation*} 
\end{proposition}
\begin{proof}
The proof of the first inequality is a special case of a more general result in \cite[Prop. 4.2]{Gmeineder2019a}. For the second one however in \cite{Smith1970} it is shown that
\begin{equation}
\|\nabla u\|_{\mathrm{L}^p(\Omega)} \leq C(\|u\|_{\mathrm{L}^p(\Omega)} + \|\bA u\|_{\mathrm{L}^p(\Omega)}).
\end{equation}
Now replace $u$ by $u-a$ and employ the first inequality.
\end{proof}
\noindent
Substantial step in when verifying these estimates relies on a norm equivalence in the virtue of Bramble-Hilbert lemma and $\mathrm{L}^2$-projections for the construction of $a$ referred to as the corrector. Hence it is vital here that $\ker (\bA)$ remains finite-dimensional. However to remain clear about the matter let us notice that the conclusions of Proposition \ref{Poinc} actually still hold when the $\bC$-ellipticity requirement is weakened. In \cite{Fuchs2011} such model of inequalities is deduced for functions defined on the unit disc.\\
We additionally give a result illustrating the crucial embedding into Besov scales which permit one to extrapolate fractional differentiability properties of $\mathrm{BV}^{\bA}$-maps. 
\begin{theorem}\label{embb}
For an open $\Omega \subset \bR^n$, any $0<s<1$ and $\bA$ being $\bR$-elliptic meaning the symbol $\bA[\xi]$ is injective for all $\xi \in \bR^n \setminus \{0\}$, there is a continuous embeddings in the norm topology $\mathrm{BV}^{\bA}(\Omega) \hookrightarrow (\mathrm{B}^{s}_{1,\infty})_{\mathrm{loc}}(\Omega)$.
\end{theorem}
\begin{proof}
First we argue for $u \in C^{\infty}_c(\bR^n; V)$. Let $B \Subset \Omega$ be a ball. The functional relation is fundamentally incited bu the Nikolskij-type estimate \cite[Lemma 4.7]{Gmeineder2019a} which asserts that with the hypotheses of Theorem \ref{embb} for a fixed $R > 0$ there exists a constant $c_{s, R} > 0$ such that $\forall v \in \mathrm{W^{\bA, 1}}(\bR^n)$ with $\supp\, v \subset B_R(0)$
\begin{equation}
\|v(\cdot + h) - v \|_{\mathrm{L}^1(\bR^n)} \leq c_{s, R}\| \bA v\|^p_{\mathrm{L}^1(B_R(0))}|h|^{sp}
\end{equation}
for any $p < \frac{n}{n-1+s}$. It then follows immediately that $|u|_{\mathrm{B}^{s}_{1,\infty}(B)} \leq c_{s, R}\| \bA u\|_{\mathrm{L}^1(\bR^n)}$.\\
For a general $u\in \mathrm{BV^{\bA}}(\Omega)$ take the global extension $E_{\Omega}u$ (Theorem \ref{compactn} (d)) and find a sequence $\{v_k\} \in (C^{\infty}\cap \mathrm{BV}^{\bA})(\bR^n)$ such that $v_k \rightarrow E_{\Omega}u$ strictly. Now for $\psi \in C^{\infty}_c(\bR^n)$ such that $\chi_{B} \leq \psi \leq \chi_{2B}$ define $\psi_k(x) = \psi(2^{-k}x)$ and $u_k := \psi_k v_k \in C^{\infty}_c(\bR^n)$. Observe that $u_k \rightarrow E_{\Omega}u$ strictly as well and in particular $u_k  \rightarrow u$ pointwise a.e. up to extraction of a subsequence. Hence
\begin{equation*}
\begin{aligned}
|u|_{\mathrm{B}^{s}_{1,\infty}(B)} &\leq \liminf_{k\rightarrow \infty} |u_k|_{\mathrm{B}^{s}_{1,\infty}(B)} \leq C \liminf_{k\rightarrow \infty} \|\bA u_k\|_{\mathrm{L}^1(\bR^n)}\\
&\leq C\liminf_{k\rightarrow \infty} \|(E_{\Omega} u)\|_{\mathrm{BV}^{\bA}(\bR^n)} \leq C\|u\|_{\mathrm{BV}^{\bA}(\Omega)}
\end{aligned}
\end{equation*}
and this concludes the proof.
\end{proof}
\noindent
In actuality one can strive for a sharper bound in the fractional scales of such kind. For instance at the level of global inequalities Van Schaftingen \cite[Thm 8.1]{VanSchaftingen2011} points out that the necessary additional assumption is the $cocancelling$ property of the operator $\bA$. Simultaneously for domains which is quintessential from the regularity perspective,  from \cite[Thm 1.3]{Gmeineder2019a} we see that taking $\bA$ to have finite-dimensional nullspace one obtains the embedding
\begin{equation*}
\mathrm{W}^{\bA, 1}(B) \hookrightarrow \mathrm{W}^{s,\frac{n}{n-1+s}}(B)
\end{equation*}
for all balls $B \subset \bR^n$. Precisely, the authors influenced by Van Schaftingen's theoy provide an equivalence relation between such estimates and the $\dim \ker(\bA; \Omega) < \infty$ condition.
\section{Relaxation of the functional and the main regularity theorem}\label{Sectionrel}
\noindent
Throughout there rest of the paper we set once and for all $\bA$ to be a $\bC$-$elliptic$ operator of order one. In order to regard corresponding integral functionals on $\mathrm{BV}^{\bA}$ maps we ought to constitute an integrand which could act on measures. It is normally achieved by introducing so called perspective function which is in a way a homogenisation of the original integrand. We allow ourselves though to diverge our attention from that and instead plunge right into the desired formula. Thus motivated by the Lebesgue-Radon-Nikodym decomposition w.r.t. the Lebesgue measure (\ref{Lebdec}) in the spirit of BV and BD theory \cite{Gmeineder2019d}:
\begin{equation*}
	\bA u = \bA^a u + \bA^s u = \bA [\nabla] u + \frac{\mathrm{d}\bA^s u}{\mathrm{d}|\bA^s u|}\mathrm{d}|\bA^s u|,
\end{equation*}
we contemplate the following relaxed functional on maps $u$ belonging to $\mathrm{BV}^{\bA}(\Omega)$
\begin{equation*}
\begin{aligned}
\mathscr{F}_{u_0}^*[u;\Omega] = \int_{\Omega}^{} f(\bA [\nabla] u) \hspace*{0.1cm}\mathrm{d}\mathscr{L}^{n} &+ \int_{\Omega}^{} f^{\infty}\Big(\frac{\mathrm{d}\bA^s u}{\mathrm{d}|\bA^s u|}\Big) \hspace*{0.1cm}\mathrm{d}|\bA^s u|\\
&+ \int_{\partial \Omega}f^{\infty}\Big(\mathrm{Tr}(u-u_0) \otimes_{\bA} \nu_{\partial\Omega}\Big) \hspace*{0.1cm}\mathrm{d} \mathscr{H}^{n-1}.
\end{aligned}
\end{equation*}
where $f^{\infty}(x) := \lim_{t\rightarrow0+}tf(z/t)$ is the usual recession function that captures the behaviour of the integrand at infinity. The above integral representation is inspired by the contributions in \cite{Anzellotti1985, Giaquinta1979, Demengel1984}. The properties (1)-(3) of $f$ mentioned at the very beginning of the Section \ref{intro} ensure that this functional is well-defined and bounded below. Crucially we have got to emphasize that the existence of the last term in the above definition is attributed to the trace class of $\mathrm{BV}^{\bA}$ maps.\\
To comprehend the continuity notions within such constructions we invoke a classical result due to Reshetnyak \cite{Reshetnyak1968a}:
\begin{theorem}$\mathrm{(Reshetnyak)}$\label{Reshtk}
Let $\mu,\mu_1,\mu_2,\dots \in \cM(\Omega; W)$ be a sequence of signed Radon measures taking values in a closed and convex cone $C \subset W$ and $g:C \rightarrow [0, \infty]$ a measurable function. Then we have the following:
\begin{enumerate}
\item if $\mu_k \rightharpoonup^{*} \mu$ and $g$ is lower semicontinuous, convex and 1-homogeneous, then
\begin{equation*}
\int_{\Omega} g\Big(\frac{\mathrm{d}\mu}{\mathrm{d}|\mu|}\Big)\mathrm{d}|\mu| \leq \liminf_{k\rightarrow \infty} \int_{\Omega} g\Big(\frac{\mathrm{d}\mu_k}{\mathrm{d}|\mu_k|}\Big)\mathrm{d}|\mu_k|.
\end{equation*}
\item If in turn $g$ is continuous, 1-homogeneous and $\mu_k \rightarrow \mu$ strictly, then
\begin{equation*}
\lim_{k\rightarrow \infty}\int_{\Omega} g\Big(\frac{\mathrm{d}\mu_k}{\mathrm{d}|\mu_k|}\Big) \mathrm{d}|\mu_k|=  \int_{\Omega} g\Big(\frac{\mathrm{d}\mu}{\mathrm{d}|\mu|}\Big)\mathrm{d}|\mu|.
\end{equation*}
\end{enumerate}
\end{theorem}
\noindent
As a corollary given $f:W \rightarrow \bR$ with properties (1)-(3) in Section \ref{intro} if we take the measure functional
\begin{equation*}
f[\mu; A] := \int_{A} f\Big(\frac{\mathrm{d}\mu}{\mathrm{d}\mathscr{L}^n}\Big) \hspace*{0.1cm}\mathrm{d}\mathscr{L}^n + \int_{A} f^{\infty}\Big(\frac{\mathrm{d}\mu^s}{\mathrm{d}|\mu^s|}\Big) \hspace*{0.1cm}\mathrm{d}|\mu^s| \hspace*{0.5cm} \text{for $\mu \in \cM(\Omega; W)$ , $A \in \mathscr{B}(\Omega)$},
\end{equation*}
then it follows that for for all $u,u_1,u_2, \dots \in \mathrm{BV}^{\bA}(\Omega)$ such that $u_k \stackrel{\langle \cdot \rangle}{\longrightarrow} u$ in $\mathrm{BV}^{\bA}(\Omega)$ we have that $f[\bA u_k; \Omega] \rightarrow f[\bA u; \Omega]$.
\begin{definition}(Generalised minimisers)\label{gmdef}
We say that $u \in \mathrm{BV}^{\bA}(\Omega)$ is a generalised minimiser of the Dirichlet problem (\ref{Dirprob}) if
\begin{equation*}
\mathscr{F}^*_{u_0}[u] \leq \mathscr{F}^*_{u_0}[v] \hspace*{1cm} \forall v \in  \mathrm{BV}^{\bA}(\Omega).
\end{equation*}  
\end{definition}
\noindent
Thanks to the lower semicontinuity of the relaxed functional as well as the required compactness within the $\mathrm{BV}^{\bA}$ spaces we are now able conclude the existence of generalised minimisers. A more general result of the sort is proved in \cite{Breit2010} where the integrand $f$ is assumed to be $\bA$-quasiconvex. We shall state a special version that depicts well the underlying mechanism. 
\begin{theorem}
Let $\Omega \subset \bR^n$ be a bounded domain with $\partial \Omega$ Lipschitz. For an integrand $f:W \rightarrow \bR$ specified in the introductory Section \ref{intro} and a boundary datum $u_0 \in \mathrm{W}^{\bA,1}(\Omega)$ there exists a generalised minimiser $u \in \mathrm{BV}^{\bA}(\Omega)$ as in the Definition \ref{gmdef} which is moreover a weak* limit of a minimising sequence of elements in $\mathscr{D}_{u_0}$. Furthermore in that case
\begin{equation*}
	\inf_{\mathscr{D}_{u_0}}\mathscr{F} = \min_{\mathrm{BV}^{\bA}(\Omega)} \mathscr{F}^*_{u_0} = \mathscr{F}^*_{u_0}[u].
\end{equation*} 
This is often referred to as the 'no gap' result.
\end{theorem}  
\begin{proof}
Take a ball $B \Supset \Omega$. By surjectivity of the trace operator we can find $v \in \mathrm{W}^{\bA, 1}(B \setminus\overline{\Omega})$ such that $\mathrm{Tr}_{\partial B}(v) = 0$ and $\mathrm{Tr}_{\partial\Omega}(v) = u_0$. Furthermore the glueing theorem \ref{glue} $\tilde{v} := v \sqcup u \in BV^{\bA}(B)$ and we henceforth fix the notation for such glueing. The measure functional given by
\begin{equation*}
\begin{aligned}
f[\bA; B](\tilde{v}) &:=  \int_{\Omega}^{} f(\bA [\nabla] u) \hspace*{0.1cm}\mathrm{d}\mathscr{L}^{n} + \int_{\Omega}^{} f^{\infty}\Big(\frac{\mathrm{d}\bA^s u}{\mathrm{d}|\bA^s u|}\Big) \hspace*{0.1cm}\mathrm{d}|\bA^s u|\\
&+\int_{\partial \Omega}f^{\infty}\Big(\mathrm{Tr}(u-u_0) \otimes_{\bA} \nu_{\partial\Omega}) \hspace*{0.1cm}\mathrm{d} \mathscr{H}^{n-1} + \int_{B \setminus\overline{\Omega}}^{} f(\bA v) \hspace*{0.1cm}\mathrm{d}\mathscr{L}^{n}
\end{aligned}
\end{equation*}
is well-defined and bounded below on $\mathrm{BV}^{\bA}(B)$. Next, extract a minimising sequence $\{u_j\} \in  \mathrm{BV}^{\bA}(\Omega)$ of $\mathscr{F}^*_{u_0}$. By construction of the functional it also follows that the extension $\tilde{u_j}$ is a minimising sequence of $f[\bA; B](\tilde{u})$ on $\mathrm{BV}^{\bA}(B)$. Now compactness of the $\mathrm{BV}^{\bA}$ space up to a subsequence $u_j \rightarrow^*w$ in $\mathrm{BV}^{\bA}(\Omega)$ and thus $\tilde{u_j}\rightarrow^*\tilde{w}$ in $\mathrm{BV}^{\bA}(B)$. By Reshetnyak's lower semicontinuity $f[\bA; B](\tilde{w}) \leq \liminf_{j\rightarrow \infty}f[\bA; B] (\tilde{u_j}) = \inf f[\bA; B](\mathrm{BV}^{\bA}(B))$ and so $\tilde{w}$ minimises $f[\bA; B]$. Since
\begin{equation}\label{proof2.3}
f[\bA; B](\tilde{u}) = \mathscr{F}^*_{u_0}[u] + \int_{B \setminus\overline{\Omega}}^{} f(\bA v) \hspace*{0.1cm}\mathrm{d}\mathscr{L}^{n}
\end{equation}
$\tilde{w}$ minimises $f[\bA; B]$ if and only if $w$ minimises $\mathscr{F}^*_{u_0}$  which justifies the first part of the theorem. Because $\mathscr{F}^*_{u_0}|_{\mathscr{D}_{u_0}} = \mathscr{F}$ it immediately implies that $\inf \mathscr{F}^*_{u_0}[\mathrm{BV^{\bA}}(\Omega)] \leq \inf \mathscr{F}[\mathscr{D}_{u_0}]$. With regards to the reverse inequality  by the trace preserving smooth approximation there is $\{u_j\} \in u_0 + C^{\infty}_c(\Omega)$ such that $\tilde{u_j} \rightarrow \tilde{u}$ strictly in $\mathrm{BV^{\bA}}$. Then invoking Reshetnyak's continuity theorem
\begin{equation*}
f[\bA; B](\tilde{u}) = \lim_{j\rightarrow\infty} f[\bA; B](\tilde{u_j}) \geq \inf \mathscr{F}[\mathscr{D}_{u_0}] + \int_{B \setminus\overline{\Omega}}^{} f(\bA v) \hspace*{0.1cm}\mathrm{d}\mathscr{L}^{n}
\end{equation*}
since $\mathrm{Tr}_{\partial\Omega}(\tilde{u_j}) = \mathrm{Tr}_{\partial\Omega}(u_0)$ on $\partial\Omega$. On the other hand by (\ref{proof2.3})
\begin{equation}
f[\bA; B](\tilde{u}) = \min \mathscr{F}^*_{u_0}[\mathrm{BV^{\bA}}(\Omega)] + \int_{B \setminus\overline{\Omega}}^{} f(\bA v) \hspace*{0.1cm}\mathrm{d}\mathscr{L}^{n}.
\end{equation}
Subtracting the two equations from each other yields $\inf \mathscr{F}^*_{u_0}[\mathrm{BV^{\bA}}(\Omega)] \geq \inf \mathscr{F}[\mathscr{D}_{u_0}]$ settling the proof.
\end{proof}
\noindent
According to conventions in the literature the collection of all generalised minimisers of such sort is defined by $\mathrm{GM}(\mathscr{F}; u_0)$. Let us mention also that even though in a number of instances strict convexity of the integrand would give uniqueness of minimising elements, the proposed generalised minima are often not unique. Partially this is on the account of the singular part $\bA^s u$ occurrence especially the non-strictly convex $|\bA^s u|(\bar{\Omega})$ and this can be seen already for low dimensions e.g. in the Santi's example \cite{Santi1972}.

\section{Towards the proof of regularity}\label{Section4}
\subsection{Viscosity type argument}\label{viscsection}
Subsequently we come to prove the central theorem of our analysis displaying the local Sobolev regularity of the generalised minimisers: Theorem \ref{thmreg}. As briefly outlined in the introduction the higher Sobolev regularity is going to be extracted by means of a perturbed Euler-Lagrange equation. The primary task is to utilise an appropriate minimising sequence whose $\bA$-gradients are locally uniformly bounded in $\mathrm{L}^p(\Omega; V)$. Principally for convex integrands satisfying \emph{p-growth} for $1<p<2$ it is achieved through the \emph{vanishing viscosity} approach and is essentially carried out by adding Dirichlet energies to the original functional, that is, one considers 
\begin{equation*}
\mathscr{F}_j[u] := \mathscr{F}[u] + \frac{1}{2j}\| \nabla u \|^2_{\mathrm{L^2(\Omega)}}. 
\end{equation*}
Thanks to the Korn-type inequality $\mathscr{F}_j$ is well-defined on $\mathrm{W^{1, 2}}(\Omega)$ and the second term gives the functional a boost in ellipticity. Moreover each $\mathscr{F}_j$ is strictly convex, there exists a unique minimiser and the entire collection of these minima $\{v_j\} \subset \mathrm{W}^{1,p}(\Omega; V)$ constitutes the desired sequence. Eventually uniform estimates $\{v_j\}$ would be reflected in the original minimiser $v$. However translating this argument to $f$ as asserted in Theorem \ref{thmreg} the linear growth causes a couple of uncertainties. Among them is that the appearance of singular part $\bA^s u$ in $\mathscr{F}^*_{u_0}$ is responsible for nonuniqueness of generalised minimisers. Thus even though an adaptation of the above method could potentially work, it would only give the regularity for one specific minimiser. In consequence there is a possibility of some considerably irregular generalised minima being present as such instances have been observed for in BV \cite{Beck2013} as well as in BD \cite{Gmeineder2019}. Striving for a more universal outcomes we shall instead take on a different approach that will yield regularity for all minimisers. To a great degree it relies on an application of the Ekeland variational principle and instead producing for any generalised minimiser $u \in \mathrm{GM}(\mathscr{F}; u_0)$ a perturbed minimising sequence $\{u_j\}$ which eventually converges to $u$ \emph{weakly*} in $\mathrm{BV}^{\bA}(\Omega)$. As elaborated in the $\mathrm{BV}$ context the negative Sobolev space $\mathrm{W}^{-1,1}$ introduced in \cite{Beck2013} and rectified in \cite{Gmeineder2020BD} fits the difference quotient technique well. Despite that, mimicking the same procedure here would amount an expression
\begin{equation*}
 \int_{\Omega} \frac{|\eta(x)\Delta_{i,h}u_j(x)|^2}{(1+|\bA u_j(x)|^2)^{\frac{1}{2}}} \hspace*{0.1cm}\mathrm{d} x
\end{equation*}
for a cut-off map $\eta$, which in view of Ornstein's noninequality and no reassurance of the full gradients of $\mathrm{BV}^{\bA}$maps existing, spark a main difficulty in framing a respective bound. On the other hand taking on fractional estimates which stem from the embedding $\mathrm{BV}^{\bA}(\Omega) \hookrightarrow (\mathrm{B}^{s}_{1,\infty})_{\mathrm{loc}}(\Omega)$ for $0 < s <1$ will avoid having to at some stage control the $\mathrm{L}^1$-norm of the full gradients of $\{u_j\}$. Though this operation in turn compels the ellipticity range to be reduced to $1 < a < 1 + \frac{1}{n}$. Regarding the encompassing space for the Ekeland principle as advertised in the introduction is going to be expressed by $\mathrm{L}^1(\Omega; V)$. Our intention of proposing this very space is predominantly because we seek to set up sufficiently weak norms to be compatible with the initially accounted features of the minimising sequence. Although the full gradients of generalised minimisers are not known in advance to exist as finite Radon measures, the selection of $\mathrm{L}^{1}(\Omega; V)$ along with fractional estimates will neutralise the first order terms and rule out the presence of the difference quotient in the Euler-Lagrange equation. Nevertheless to remain scrupulous, because of $\mathrm{L}^1(\Omega; V) \hookrightarrow \mathrm{W}^{-1,1}(\Omega)$ embedding, the subsequent stabilisation procedure may well be reiterated in $\mathrm{W}^{-1,1}$. Even though this forms a link with the $\mathrm{BV}$ framework in the current scenario the cancellation of partial derivatives by means of the negative Sobolev norm seems a little superfluous. Attempting to illustrate the emerging task if we select a perturbation space $Y$ the crucial step is to control terms of the form
\begin{equation*}
\|v_j(\cdot+he_i)-v_j\|_Y
\end{equation*}
where $v_j$ is a specific map the Euler-Lagrange equation is being tested with. By letting $Y=\mathrm{L}^1$ and inserting a fractional power of parameter $h$ will then  yield a uniform bound after invoking the assertions of Thereom \ref{embb}. Nevertheless in comparison with the Lipschitz dual $(\mathrm{W}^{1, \infty}_0(\Omega))^*$ utilised in \cite{Gmeineder2019}, which also works equally in this setting, the Lebesgue space displays an arguably transparent structure and gives a slightly more accurate insight. All in all $\mathrm{L}^1(\Omega; V)$ perturbations arise as the optimal solution and it really exposes the strenght of fractional estimates. This will be more evident later when carrying out the computations. The underlying ideas for our argumentation are inspired by papers of Beck and Schmidt \cite{Beck2013} as well as Seregin \cite{Seregin1993b, Seregin1994a} and Bildhauer \cite{Bildhauer2003}. Let us begin by recalling the relevant results and setting up the core constructions.
\begin{theorem}{\normalfont{(Ekeland Variational Principle)}}\label{eke}
Let $(X,d)$ be a complete metric space and let $F:X\rightarrow [-\infty,+\infty)$ be a bounded below functional which is lower semicontinuous on the metric topology. If for a fixed $\epsilon>0$ there is some $u \in X$
\begin{equation*}
F(u) \leq \inf_{X} F + \epsilon
\end{equation*}
holds, then on can find a $v \in X$ such that
\begin{enumerate}
\item $d(u, v) \leq \sqrt{\epsilon}$
\item $F(v) \leq F(u)$
\item $F(v) \leq F(u) + \sqrt{d(v, w)}$ for all $w \in X$.
\end{enumerate}
\end{theorem}
\noindent
A matching Ekeland functional $F$ will stem from the following generalisation of the lemma in \cite{Gmeineder2019}:
\begin{proposition}\label{growth1}
Let $p>1$ and let $g: \bR^{n \times n} \rightarrow \bR$ be a convex function with p-growth:
\begin{equation}
c|z|^p - \theta \leq g(z) \leq C(1+|z|^p)
\end{equation}
for all $z \in\bR^{n \times n}$ and some constants $\theta, c, C >0$. For a given boundary datum $u_0 \in \mathrm{W}^{1, p}(\Omega; V)$, the integral functional
\begin{equation}
G[u] :=
\begin{cases}
\int_{\Omega}^{} g(\bA u) \hspace*{0.1cm}\mathrm{d}\mathscr{L}^n, \quad \text{if $ u \in u_0 + \mathrm{W}^{1, p}_0(\Omega; V)$}\\
+\infty, \quad \quad \quad \quad \quad \,  \text{if $u \in  \mathrm{L}^{1}(\Omega; V) \setminus  (u_0 + \mathrm{W}^{1, p}_0(\Omega; V))$}
\end{cases}
\end{equation}
is lower semicontinuous with respect to the norm topology on $\mathrm{L}^{1}$.
\end{proposition}
\begin{proof}
Take a sequence $\{u_j\}$ in $\mathrm{L}^{1}$ such that $u_j \rightarrow u$  in the norm topology of $\mathrm{L}^{1}$. In the case of $\liminf_{j\rightarrow \infty}G[u_j] = \infty$ the property is clear. Hence it suffices to study the case when $\liminf_{j\rightarrow \infty}G[u_j] < \infty$ and find a subsequence $\{u_{j_k}\}$ with $\lim_{k\rightarrow \infty}G[u_{j_k}] = \liminf_{j\rightarrow \infty}G[u_j]$. The construction of $G$ forces $\{u_{j_k}\}$ to be in $u_0 + \mathrm{W}^{1, p}_0(\Omega; V)$. The $p$-growth of $g$ implies further that $\{\bA u_{j_k}\}$ is bounded in $\mathrm{L}^p(\Omega; V)$. Using Proposition \ref{Poinc} we conclude further that $\{u_{j_k}\}$ is bounded in $\mathrm{W}^{1, p}(\Omega; V)$ and hence passing to a subsequence $u_{j_{k_m}} \rightharpoonup v \in \mathrm{W}^{1, p}(\Omega; V)$. By Rellich-Kondrachov this convergence can be regarded strongly in $\mathrm{L}^p$. Now since $\mathrm{L}^p(\Omega; V) \hookrightarrow \mathrm{L}^{1}(\Omega; V)$ it follows that $u = v$. Finally by convexity, lower semicontinuity for $p$-growth integrands
\begin{equation*}
G[u] = G[v] \leq \liminf_{m\rightarrow \infty} G[u_{j_{k_m}}] = \liminf_{m\rightarrow \infty} G[u_{j_{k_m}}] =  \liminf_{k\rightarrow \infty} G[u_{j_k}]
\end{equation*} 
and this finishes the proof.
\end{proof}
\begin{lemma}\label{growth2}
Let $f$ be the integrand given in the introductory Section 2.1. Then for any $\alpha >0$ we can find $c_\alpha, C_{\alpha}, \theta>0$ such that 
\begin{equation*}
c_{\alpha}|z|^2- \theta \leq f(z) + \alpha|z|^2 \leq C_{\alpha}(1+|z|^2)
\end{equation*}
for all $z \in \bR^{n \times n}.$
\end{lemma}
\begin{proof}
Clear.
\end{proof}
\noindent
\underline{\textit{Construction of the viscosity approximation sequence}}:\\ 
We fix a generalised minimiser $u \in \mathrm{BV^{\bA}}(\Omega)$. Then by the trace preserving smooth approximation \cite{Breit2010} we can find a sequence $v_j \in \mathscr{D}_{u_0}$ such that $v_j \stackrel{\langle \cdot \rangle}{\longrightarrow} u$ area-strictly in $\mathrm{BV}^{\bA}(\Omega)$. Combining Theorem \ref{Reshtk} and the fact that $\mathscr{F}^*_{u_0}|_{ \mathscr{D}_{u_0}} = \mathscr{F}$ one concludes that $\lim_{j\rightarrow\infty} \mathscr{F}[v_j] = \mathscr{F}^*_{u_0}[u] = \inf_{\mathscr{D}_{u_0}} \mathscr{F}$. Extracting a subsequence and omitting relabelling
\begin{equation}
 \mathscr{F}[v_j] \leq \inf_{\mathscr{D}_{u_0}} \mathscr{F} + \frac{1}{8j^2}.
\end{equation} 
Due to the boundary values, the surjectivity of trace operator $\mathrm{Tr}$ by mollification we find a $v^{\partial\Omega}_j \in \mathrm{W}^{1,2}(\Omega; V)$ for each $j$ such that
\begin{equation}
\|v^{\partial\Omega}_j - u_0\|_{\mathrm{W}^{\bA,1}} \leq \frac{1}{8 \mathrm{Lip}(f)j^2}.
\end{equation}
Setting $\mathscr{D}_j := v^{\partial\Omega}_j + \mathrm{W}^{1,2}_0(\Omega; V)$ we further deduce that there exists $\tilde{u_j} \in \mathscr{D}_j$ with
\begin{equation}
\|v_j - u_0 - (\tilde{u_j}-v^{\partial\Omega}_j)\|_{\mathrm{W}^{\bA, 1}} \leq \frac{1}{8 \mathrm{Lip}(f)j^2}
\end{equation}
on the account that $v_j - u_0 \in \mathrm{W}^{\bA, 1}_0(\Omega)$. In summary it yields
\begin{equation}\label{Lipbound}
\|v_j - \tilde{u_j}\|_{\mathrm{W}^{\bA, 1}} \leq \frac{1}{4 \mathrm{Lip}(f)j^2}	
\end{equation}
for all $j \in \bN$. By considering the smooth approximation of $u_0 + \mathrm{W}^{1,2}_0(\Omega; V)$ in $\mathscr{D}_{u_0}$ and the above array of inequalities we perform the following computation: for any $\phi \in \mathrm{W}^{1,2}_0(\Omega; V)$
\begin{equation*}
\begin{aligned}
\inf \mathscr{F}[\mathscr{D}_{u_0}] &\leq \mathscr{F}[u_0 + \phi] = \mathscr{F}[u_0 + \phi] -\mathscr{F}[v^{\partial\Omega}_j + \phi] + \mathscr{F}[v^{\partial\Omega}_j + \phi]\\
&\leq \mathrm{Lip}(f)\|\bA(u_0 - v^{\partial\Omega}_j)\|_{\mathrm{L}^1(\Omega)} + \mathscr{F}[v^{\partial\Omega}_j + \phi] \leq\frac{1}{8j^2} + \mathscr{F}[v^{\partial\Omega}_j + \phi]\\
\Rightarrow \hspace*{1cm} & \inf_{\mathscr{D}_{u_0}} \mathscr{F} \leq \frac{1}{8j^2} + \inf_{\mathscr{D}_{j}} \mathscr{F}
\end{aligned}
\end{equation*}
by taking infimum over all such $\phi$. Similarly using the same argument gives
\begin{equation}\label{inf F}
\mathscr{F}[\tilde{u_j}] \leq \inf_{\mathscr{D}_{u_j}} \mathscr{F} + \frac{1}{2j^2}.
\end{equation}
Let us now define the requisite functionals to be involved in the implementation of the Ekeland principle: for $j \in \bN$ define
\begin{equation}\label{f_j}
f_j(\xi) := f(\xi) + \frac{1}{2j^2\mathrm{A}_j}(1+|\xi|^2) \hspace*{1cm} \text{where} \hspace*{1cm} \mathrm{A}_j := 1 + \int_{\Omega}^{} (1 + |\bA(\tilde{u_j})|^2)  \hspace*{0.1cm}\mathrm{d} x.
\end{equation}
Moreover we extend the integral functionals to the space $\mathrm{L}^{1}(\Omega; V)$:
\begin{equation}\label{F_j}
\mathscr{F}_j[u] =
\begin{cases}
\int_{\Omega}^{} f_j(\bA u) \hspace*{0.1cm}\mathrm{d} x, \quad \text{if $ u \in \mathscr{D}_j$}\\
+\infty, \quad \quad \quad \quad \quad \,  \text{if $u \in  \mathrm{L}^{1}(\Omega; V) \setminus  \mathscr{D}_j$}.
\end{cases} 
\end{equation}
By construction $\mathscr{F}_j$ outside $\mathscr{D}_j$ hence in order to check lower semicontinuity with respect to the norm topology on $\mathrm{L}^{1}$ it suffices to focus on the subdomain $\mathscr{D}_j$. This however is readily verified having Proposition \ref{growth1} and Lemma \ref{growth2} at hand for special cases $p = 2$ and $g = f_j$.\\
Invoking the Ekeland Principle Theorem \ref{eke} yields a sequence $\{u_j\}$ in $\mathrm{L}^{1}$ that satisfies:
\begin{equation}\label{eke2}
\|u_j - \tilde{u_j}\|_{\mathrm{L}^{1}} \leq \frac{1}{j}, \hspace*{1cm} \mathscr{F}_j[u_j] \leq \mathscr{F}_j[w] + \frac{1}{j}\|u_j - w\|_{\mathrm{L}^{1}}
\end{equation}
and this holds for all $w \in \mathrm{L}^{1}(\Omega; V)$.\\
\newline
\underline{\textit{We claim that $\{u_j\}$ is uniformly bounded in $\mathrm{W}^{\bA,1}(\Omega)$}}. \label{unifW}($\ast$)\\
Indeed by regarding the two inequalities above together with the linear growth of $f$, because of the embedding $\mathrm{W}^{1,2}(\Omega; V) \hookrightarrow \mathrm{L}^{1}(\Omega; V)$ there is a constant $\Theta > 0$ for which we get   
\begin{equation*}
\begin{aligned}
\int_{\Omega}^{} |\bA u_j| \hspace*{0.1cm}\mathrm{d}\mathscr{L}^n &\leq \mathscr{F}[u_j] + \Theta \leq \mathscr{F}_j[u_j] + \Theta \leq \mathscr{F}_j[u_j] \leq \mathscr{F}_j[w] + \frac{1}{j}\|u_j - \tilde{u_j}\|_{\mathrm{L}^{1}} + \Theta\\
& \leq \inf \mathscr{F}_j[\mathrm{L}^{1}(\Omega; V)] + \frac{2}{j^2} + \Theta < \infty
\end{aligned}
\end{equation*}
where in the penultimate inequality we have used equation (\ref{inf F}) and that $\mathscr{F}_j$ is bounded below by construction.\\
In order to justify regularity of minimisers it is customary to test the induced Euler-Lagrange equation with a suitable test map. The arising circumstances allow the following variant:
\begin{lemma}\normalfont{{(Perturbed Euler-Lagrange equation)}}
For all $j \in \bN$ and all  $\phi \in \mathrm{W}^{1,2}_0(\Omega; V)$ it holds that
\begin{equation}\label{E-L}
\Big| \int_{\Omega}^{} \langle \nabla f_j(\bA u_j), \bA \phi \rangle \hspace*{0.1cm}\mathrm{d}\mathscr{L}^n \Big| \leq \frac{1}{j} \|\phi\|_{\mathrm{L}^{1}}.
\end{equation}
\end{lemma}
\begin{proof}
Fix $\epsilon>0$. Then testing the inequality in (\ref{eke2}) with $u_j \pm \epsilon \phi$ so that  
\begin{equation*}
\begin{aligned}
\mathscr{F}_j[u_j] - \mathscr{F}_j[u_j \pm \epsilon \phi] &\leq \frac{\epsilon}{j} \|\phi\|_{\mathrm{L}^{1}}\\
\Rightarrow \hspace*{1cm} \frac{\mathscr{F}_j[u_j] - \mathscr{F}_j[u_j \pm \epsilon \phi]}{\epsilon} &\leq \frac{1}{j} \|\phi\|_{\mathrm{L}^{1}}.
\end{aligned}
\end{equation*}
Notice that $u_j \pm \epsilon \phi \in \mathscr{D}_j$ as a consequence of $\{u_j\}$ being bounded in $\mathrm{W}^{\bA,1}(\Omega)$ and (\ref{F_j}). Now recalling the definition of $\mathscr{F}_j$ and letting $\epsilon \rightarrow 0$ yields (\ref{E-L}).
\end{proof}
\subsection{Proof of Theorem \ref{thmreg}}
In the upcoming section the procedure will follow the lines of \cite{Gmeineder2019} with sufficient modifications. We fix $x_0 \in \Omega$ and let $t > 0$ to be determined later. Further consider $1<a<1 + \frac{1}{n}$ and a $u \in \mathrm{BV^{\bA}}(\Omega)$ a generalised minimiser of the variational problem (\ref{Dirprob}). Let $\{u_j\}$ be the associated Ekeland sequence contemplated in the subsection devoted to viscosity sequence. Consider a cut-off map $\eta \in C^{\infty}_c(B_R(x_0))$ such that $\chi_{B_r(x_0)} \leq \eta \leq \chi_{B_R(x_0)}$ and $|\nabla \eta| \leq \frac{2}{R-r}$ where $0<r<R<\mathrm{dist}(x_0, \partial \Omega)$. We now take $\tilde{\Omega}$ a connected component of $x_0$ in $\Omega$ in order to find ourselves in the assertions of Proposition \ref{Poinc} and subsequently come up, for each $j \in \bN$, with an element $a_j \in \ker(\bA)$ satisfying
\begin{equation}\label{BoundOnPsi}
\|u_j - a_j\|_{\mathrm{L}^1(\tilde{\Omega})} \leq C \| \bA u_j \|_{\mathrm{L}^1(\tilde{\Omega})}, \hspace*{1cm} \|\nabla(u_j - a_j)\|_{\mathrm{L}^2(\tilde{\Omega})} \leq C \| \bA u_j \|_{\mathrm{L}^2(\tilde{\Omega})}.
\end{equation}
At this moment we shall conduct an argument that constitutes the essential part of the entire regularity proof. To begin with, we test the induced perturbed Euler-Lagrange equation (\ref{E-L}) with $\phi := \tau^+_{i,h}(\eta^2\tau^-_{i,h}(u_j-a_j)) = \tau^+_{i,h}(\eta^2\tau^-_{i,h}(\psi_j))$ where $|h| < \mathrm{dist}(\partial B_R(x_0), \tilde{\Omega} )$ and by construction  this is an admissible map. Thus 
\begin{equation}\label{step1}
\Big| \int_{\Omega}^{} \langle \nabla f_j(\bA u_j), \tau^+_{i,h}(\bA(\eta^2\tau^-_{i,h}(\psi_j)) \rangle \hspace*{0.1cm}\mathrm{d}\mathscr{L}^n \Big| \leq \frac{1}{j} \|\tau^+_{i,h}(\eta^2\tau^-_{i,h}(\psi_j)\|_{\mathrm{L}^{1}}
\end{equation}
and integrating parts the left hand side of the above equation combined with the product rule (\ref{prdctrul}) for $\bA$ this reads
\begin{equation*}
\begin{aligned}
&\int_{\Omega}^{} \langle \tau^+_{i, h}\nabla f_j(\bA u_j), \eta^2\tau^-_{i,h}(\bA u_j)) \rangle \hspace*{0.1cm}\mathrm{d}\mathscr{L}^n\\ 
&\leq \Big|\int_{\Omega}^{} \langle \tau^+_{i, h}\nabla f_j(\bA u_j), 2\eta\bA \eta \otimes_{\bA}\tau^-_{i,h}(\psi_j) \rangle \hspace*{0.1cm}\mathrm{d}\mathscr{L}^n \Big| + \frac{1}{j} \|\tau^+_{i,h}(\eta^2\tau^-_{i,h}(\psi_j))\|_{\mathrm{W}^{-2,1}}\\
&\leq \Big|\int_{\Omega}^{} \langle \tau^+_{i, h}\nabla f(\bA u_j), 2\eta\bA \eta \otimes_{\bA}\tau^-_{i,h}(\psi_j) \rangle \hspace*{0.1cm}\mathrm{d}\mathscr{L}^n \Big| + \frac{1}{j^2\mathrm{A}_j}\Big|\int_{\Omega}^{} \langle \tau^+_{i, h}\bA u_j, 2\eta\bA \eta \otimes_{\bA}\tau^-_{i,h}(\psi_j) \rangle \hspace*{0.1cm}\mathrm{d}\mathscr{L}^n \Big|\\
&+ \frac{1}{j} \|\tau^+_{i,h}(\eta^2\tau^-_{i,h}(\psi_j))\|_{\mathrm{L}^{1}}
\end{aligned}
\end{equation*}
where in the last inequality the definition of $f_j$ (\ref{f_j}) is being used. We will henceforth refer to the terms involved in the derived inequality by
\begin{equation}
\mathbf{I} \leq \mathbf{II} + \mathbf{III} + \mathbf{IV}.
\end{equation}
Consequently our effort are going to focus on estimating the respective terms and we shall do so one by one outlined in several steps:\\
\underline{Step 1: bound on $\mathbf{II}$}: as $\|\nabla f\|_{\mathrm{L}^{\infty}(\Omega)} \leq C$ by the initial assumptions, for any $0 < s < 1$:
\begin{equation}
\begin{aligned}
\mathbf{II} &\leq C \int_{B_R(x_0)}^{} |\tau_{i,h} \psi_j| \hspace*{0.1cm}\mathrm{d} x \leq C_sh^s \int_{B_R(x_0)}^{} \frac{|\tau_{i,h}\psi_j|}{h^s} \hspace*{0.1cm}\mathrm{d} x\\ 
&\leq C_sh^s|\psi_j|_{\mathrm{B}^s_{1,\infty}} \leq C_sh^s \|\psi_j\|_{\mathrm{W}^{\bA, 1}(\Omega)}\\  
&= C_sh^s \|u_j\|_{\mathrm{W}^{\bA, 1}(\Omega)} \leq C_sh^s\hspace*{3cm} \text{(by Theorem \ref{embb})}\\ 
\end{aligned}
\end{equation}
where have utilised the fact that $\bA \psi_j = \bA(u_j-a_j) = \bA u_j$ since $a_j \in N(\bA)$, Proposition \ref{Poinc} and the uniform boundedness \hyperref[unifW]{($\ast$)} in the last step. Here the absorbing constant $C>0$ only depends on the ball $B_R(x_0)$. \\
\underline{Step 2: bound on $\mathbf{III}$}: Firstly we employ Young's inequality and subsequently proceed in a similar fashion as in Step 1. For this we take $\delta < 1$ so that
\begin{equation}\label{termIII}
\begin{aligned}
\mathbf{III} &\leq \frac{\delta}{\mathrm{A}_jj^2} \int_{\Omega}^{} |\tau_{i,h} \bA u_j|^2 \hspace*{0.1cm}\mathrm{d} x + \frac{M_\delta}{\mathrm{A}_jj^2} \int_{\Omega}^{} |\nabla \eta \otimes_{\bA} \tau_{i,h}\psi_j|^2 \hspace*{0.1cm}\mathrm{d} x\\
&= \frac{\delta}{\mathrm{A}_jj^2} \int_{\Omega}^{} |\tau_{i,h} \bA u_j|^2 \hspace*{0.1cm}\mathrm{d} x + \frac{M_{\delta,\eta}h^2}{\mathrm{A}_jj^2} \int_{\Omega}^{} |\Delta_{i,h}\psi_j|^2 \hspace*{0.1cm}\mathrm{d} x\\
&=\frac{\delta}{\mathrm{A}_jj^2} \int_{\Omega}^{} |\tau_{i,h} \bA u_j|^2 \hspace*{0.1cm}\mathrm{d} x + \frac{M_{\delta,\eta,\tilde{\Omega}}h^2}{\mathrm{A}_jj^2} \int_{\Omega}^{} |\nabla\psi_j|^2 \hspace*{0.1cm}\mathrm{d} x\\
&\leq\frac{\delta}{\mathrm{A}_jj^2} \int_{\Omega}^{} |\tau_{i,h} \bA u_j|^2 \hspace*{0.1cm}\mathrm{d} x + \frac{M_{\delta,\eta,\tilde{\Omega}}h^2}{\mathrm{A}_jj^2} \int_{\Omega}^{} |\bA u_j|^2 \hspace*{0.1cm}\mathrm{d} x \hspace*{2cm} \text{(by Proposition \ref{Poinc})}\\
&\leq \frac{\delta}{\mathrm{A}_jj^2} \int_{\Omega}^{} |\tau_{i,h} \bA u_j|^2 \hspace*{0.1cm}\mathrm{d} x + \frac{M_{\delta, \eta, \tilde{\Omega}}h^2}{j^2} =: \mathbf{III}' + \frac{M_{\delta, \eta, \tilde{\Omega}}h^2}{j^2} \hspace*{1cm} \text{by (\ref{f_j}) and (\ref{unifW})}.
\end{aligned}
\end{equation}
\underline{Step 3: bounds on $\mathbf{IV}$}: For this term we argue similarly as for $\mathbf{II}$:
\begin{equation}\label{termIV}
\begin{aligned}
\mathbf{IV} &= \frac{h^s}{j} \Big\|\frac{\tau^+_{i,h}(\eta^2\tau^-_{i,h}(\psi_j))}{h^s}\Big\|_{\mathrm{L}^{1}(\tilde{\Omega})} \leq  \frac{h^s}{j} \Big\|\frac{\tau^+_{i,h}(\eta^2\tau^-_{i,h}(\psi_j))}{h^s}\Big\|_{\mathrm{L^1}(\tilde{\Omega})} \leq \frac{h^s}{j}|\eta^2\tau^-_{i,h}(\psi_j)|_{\mathrm{B}^s_{1,\infty}}\\
&\leq\frac{h^s}{j}C\|\eta^2\tau^-_{i,h}(\psi_j)\|_{\mathrm{W}^{\bA, 1}(\tilde{\Omega})}
\leq \frac{h^s}{j}C \|u_j\|_{\mathrm{W}^{\bA, 1}(\Omega)} \leq \frac{h^s}{j}C
\end{aligned}
\end{equation}
where the second inequality is obtained by recalling Theorem \ref{embb} and the last two by invoking (\ref{BoundOnPsi}) as well as the uniform bound \hyperref[unifW]{($\ast$)}.\\
\underline{Step 4: bounds on $\mathbf{I}$}: We estimate this term from below in the following way. 
\begin{equation}
\begin{aligned}
\mathbf{I} &= \int_{\Omega}^{} \langle \tau^+_{i, h}\nabla f_j(\bA u_j), \eta^2\tau^-_{i,h}(\bA u_j)) \rangle \hspace*{0.1cm}\mathrm{d} x = \int_{\Omega}  \Big\langle\int_{0}^{1} \frac{\mathrm{d}}{\mathrm{d} t}f_j(\bA u_j + t\tau_{i,h}\bA u_j)\hspace*{0.1cm}\mathrm{d} t, \eta^2\tau_{i,h} \bA u_j \Big\rangle  \hspace*{0.1cm}\mathrm{d} x\\
&\geq \int_{\Omega} \int_{0}^{1} \langle \nabla^2f(\bA u_j + t\tau_{i,h}\bA u_j)\eta\tau_{i,h} \bA u_j, \eta\tau_{i,h} \bA u_j \rangle \hspace*{0.1cm}\mathrm{d} t \hspace*{0.1cm}\mathrm{d} x + \frac{1}{\mathrm{A}_jj^2}\int_{\Omega}^{} |\eta\tau_{i,h} \bA u_j|^2 \hspace*{0.1cm}\mathrm{d} x\\
& \geq c_1 \int_{\Omega} \int_{0}^{1} \frac{|\tau_{i,h}\bA u_j(x)|^2}{(1+|\bA u_j +t\tau_{i,h} \bA u_j(x)|^2)^{\frac{a}{2}}} \hspace*{0.1cm}\mathrm{d} t \hspace*{0.1cm}\mathrm{d} x + \frac{1}{\mathrm{A}_jj^2}\int_{\Omega}^{} |\eta\tau_{i,h} \bA u_j|^2 \hspace*{0.1cm}\mathrm{d} x  \hspace*{0.5cm} \text{by $a$-$ellipticity$}\\
&\geq c_1 \int_{\Omega} \int_{0}^{1} \frac{|\tau_{i,h}\bA u_j(x)|^2}{(1+|\bA u_j|^2 +t\tau_{i,h} \bA u_j(x)|^2)^{\frac{a}{2}}} \hspace*{0.1cm}\mathrm{d} t \hspace*{0.1cm}\mathrm{d} x + \frac{1}{\mathrm{A}_jj^2}\int_{\Omega}^{} |\eta\tau_{i,h} \bA u_j|^2 \hspace*{0.1cm}\mathrm{d} x.
\end{aligned}
\end{equation}
Observe that for $|t|<1$ it holds that $\sqrt{(1+|x+ty|^2)} \leq C \sqrt{(1+|x|^2+|y|^2)}$ for any $x,y \in \bR^n$. Therefore additionally recalling the properties of V-function (\ref{V4}) and composing it into the above inequality ultimately gives
\begin{equation}
\begin{aligned}
\mathbf{I} & \geq \int_{\Omega} \frac{|\tau_{i,h}\bA u_j(x)|^2}{(1+|\bA u_j(x+he_i)|^2 + |\bA u_j(x)|^2)^{\frac{a}{2}}} \hspace*{0.1cm}\mathrm{d} x + \frac{1}{\mathrm{A}_jj^2}\int_{\Omega}^{} |\eta\tau_{i,h} \bA u_j|^2 \hspace*{0.1cm}\mathrm{d} x\\
&\geq \tilde{c} \int_{\Omega} \frac{|\tau_{i,h}V_t(\bA u_j(x))|^2}{(1+|\bA u_j(x+he_i)|^2 + |\bA u_j(x)|^2)^{\frac{2(1-t) + a}{2}}} \hspace*{0.1cm}\mathrm{d} x + \frac{1}{\mathrm{A}_jj^2}\int_{\Omega}^{} |\eta\tau_{i,h} \bA u_j|^2 \hspace*{0.1cm}\mathrm{d} x\\
& =: \mathbf{I}_1 + \mathbf{I}_2.
\end{aligned}
\end{equation}
Let briefly induce some notation to simply the upcoming calculation and call the denominator
\begin{equation}
\sigma_{j,h,i}(x) := \frac{1}{(1+|\bA u_j(x+he_i)|^2 + |\bA u_j(x)|^2)^{\frac{2(1-t) + a}{2}}}.
\end{equation}
To summarize the outcome thus far, by fusing the estimates from Steps 1 to 4
together
\begin{equation*}
\mathbf{I}_1 + \mathbf{I}_2 \leq C\frac{h}{j} + \mathbf{III}' + \frac{M_{\delta, \eta, \tilde{\Omega}}h^2}{j^2} + C_sh^s
\end{equation*}
with the constant $C>0$ inheriting the ones from Steps 1 and 3. By the choice of $\delta > 0$ we may absorb $\mathbf{III}'$ into $\mathbf{I}_2$ which in other words means that
\begin{equation}
\tilde{c} \int_{\Omega} |\tau_{i,h}V_t(\bA u_j(x))|^2 \sigma_{j,h,i}(x) \hspace*{0.1cm}\mathrm{d} x + \frac{1-\delta}{\mathrm{A}_jj^2}\int_{\Omega}^{} |\eta\tau_{i,h} \bA u_j|^2 \hspace*{0.1cm}\mathrm{d} x \leq C\frac{h}{j} + \frac{M_{\delta, \eta, \tilde{\Omega}}h^2}{j^2} + C_sh^s.
\end{equation}
By dividing the first term on the left hand side by $h^s$ we deduce 
\begin{equation}
\sup_{j \in \bN} \int_{\Omega}^{} \frac{|\tau_{i,h}V_t(\bA u_j(x))|^2}{h^{\frac{s}{2}}} \sigma_{j,h,i}(x) \hspace*{0.1cm}\mathrm{d} x < \infty.
\end{equation} 
Further using the Young's inequality we arrive at the following expression
\begin{equation}
\begin{aligned}
\int_{\Omega}^{} \frac{|\tau_{i,h}V_t(\bA u_j(x))|}{h^{\frac{s}{2}}} \hspace*{0.1cm}\mathrm{d} x &= \int_{\Omega}^{} \frac{|\tau_{i,h}V_t(\bA u_j(x))|^2}{h^{\frac{s}{2}}} \frac{\sqrt{\sigma_{j,h,i}(x)}}{\sqrt{\sigma_{j,h,i}(x)}} \hspace*{0.1cm}\mathrm{d} x\\ 
&\leq \int_{\Omega}^{} \frac{|\tau_{i,h}V_t(\bA u_j(x))|^2}{|h^s|} \sigma_{j,h,i}(x) \hspace*{0.1cm}\mathrm{d} x + \int_{\Omega}^{} \frac{1}{\sigma_{j,h,i}(x)} \hspace*{0.1cm}\mathrm{d} x.
\end{aligned} 
\end{equation}
Now if we impose that $a +2(1-t) \leq 1$ and recall the uniform bound \hyperref[unifW]{($\ast$)}, we notice that the left hand side of the above inequality is uniformly bounded for $\forall j \in \bN$. On the other hand by the embedding $\mathrm{B}^{\frac{s}{2}}_{1,\infty}(\Omega) \hookrightarrow \mathrm{L}^{\frac{2n}{2n-s}-\epsilon}(\Omega; V)$ (\ref{embbBes}) for $0 < \epsilon < \frac{2n}{2n-s}$ it implies that for any $B \Subset \Omega$
\begin{equation}
\begin{aligned}
\sup_{j \in \bN} \int_{B} |\bA u_j|^{(2-t)(\frac{2n}{2n-s}-\epsilon)} \hspace*{0.1cm} \mathrm{d} x &\leq \sup_{j \in \bN} \int_{\Omega} |\tau_{i,h}V_t(\bA u_j(x))|^{\frac{2n}{2n-s}-\epsilon} \hspace*{0.1cm} \mathrm{d} x  + \mathscr{L}^n(\Omega) \hspace*{0.4cm} \text{(property (\ref{V3}))}\\
& \leq \int_{\Omega}^{} \frac{|\tau_{i,h}V_t(\bA u_j(x))|}{h^{\frac{s}{2}}} \hspace*{0.1cm}\mathrm{d} x + \mathscr{L}^n(\Omega) < \infty \hspace*{1cm} \text{(Besov embedding)}.
\end{aligned}
\end{equation}
Now we choose $t,s,\epsilon > 0$ such that $(2-t)(\frac{2n}{2n-s}-\epsilon)>1$ and set $p := (2-t)(\frac{2n}{2n-s}-\epsilon)$. Notice that this is possible since we can always find $\tilde{t}>0$ with $(2-\tilde{t})\frac{2n}{2n-1}>1$. In summary
\begin{equation}
\sup_{j \in \bN} \int_{B} |\bA u_j|^p \hspace*{0.1cm} \mathrm{d} x < \infty.
\end{equation}
Finally in view of the compactness result (\ref{compactn}) and the bound (\ref{unifW}) up to a subsequence $u_j \rightharpoonup^{*} w$ in $\mathrm{BV}^{\bA}$ for some $w \in \mathrm{BV}^{\bA}(\Omega)$ and therefore it is also the case that $u_j \rightarrow w$ in $\mathrm{L}^{1}$. Simultaneously as $\{u_j\}$ is the viscosity approximating sequence invoking (\ref{Lipbound}) as well as (\ref{eke2}) imply $u = w$. On the other hand by adhering to Poincar\'e inequality (\ref{Poinc}), $(u_j-a_j)|_B$ is bounded uniformly in $\mathrm{W}^{1, p}(B; V)$. As $\ker(\bA)$ is finite-dimensional this implies that $\{a_j\}$ are uniformly bounded in $\mathrm{W}^{1, p}(B; V)$. Hence in particular by reflexivity of the very space it holds that $u_j \rightharpoonup v$ in $\mathrm{W}^{1, p}(B; V)$ for some $v \in \mathrm{W}^{1, p}(B; V)$. Thus combining the two limits the equality holds $u|_B = v$ and the proof is complete.
\newline
\qed

\Addresses

\end{document}